\documentclass[a4paper,10pt]{article}

\usepackage[utf8]{inputenc}

\usepackage{amssymb}
\usepackage{amsmath}

\usepackage{fancyhdr}
\usepackage{amsthm}
\usepackage{color}
\newtheorem{Thm}{Theorem}

\newtheorem{Lem}{Lemma}

\theoremstyle{definition}

\theoremstyle{remark}
\newtheorem{Rem}{Remark}[section]

\newtheorem{Exe}{Example}[section]

\DeclareMathOperator*{\E}{\mathbb{E}}
\DeclareMathOperator*{\argmax}{argmax}

\title{Nonexistence of fractional Brownian fields\\ indexed by cylinders} 
\author{Nil Venet\footnote{Institut de Mathématiques de Toulouse, Université Paul Sabatier, 118 Route de Narbonne, 31400 Toulouse,  nil.venet@math.univ-toulouse.fr}}
\begin{document}	
	\maketitle
	\begin{abstract}
		We show in this paper that there exists no $H$-fractional Brownian field indexed by the cylinder $\mathbb{S}^1 \times ]0,\varepsilon[$ endowed with its product distance $d$ for any $\varepsilon>0$ and $H>0$. This is equivalent to say that $d^{2H}$ is not a negative definite kernel, which also leaves us without a proof that many classical stationary kernels, such that the Gaussian and exponential kernels, are positive definite kernels -- or covariances -- on the cylinder.
		
		We generalise this result from the cylinder to any Riemannian Cartesian product with a minimal closed geodesic. We also investigate the case of the cylinder endowed with a distance asymptotically close to the product distance in the neighbourhood of a circle. 
		
		Another consequence is the discontinuity of the set of $H$ such that $d^{2H}$ is negative definite with respect to the Gromov-Hausdorff convergence on compact metric spaces.
		
		These results extend our comprehension of kernel construction on metric spaces, and in particular call for alternatives to classical kernels to allow for Gaussian modelling and kernel method learning on cylinders.
	\end{abstract}
	\section{Introduction} \label{sec:introduction}
	The study of fractional random processes has been a very active topic since the article of Mandelbrot and Van Ness on fractional Brownian motion~\cite{MandelbrotVanNess}, from which they have proven to be major random models in a variety of applications. In order to model geological phenomena Mandelbrot considered in \cite{Mandelbrotfield}  fractional Brownian motion indexed by the Euclidean spaces. In \cite{istas2011manifold}, Istas stresses out the need for fractional random fields indexed by nonflat spaces and defines the $H$-fractional Brownian field indexed by any metric space. It is a natural generalisation of the classical fractional Brownian motion from which it inherits key properties such as stationary increments with respect to the isometry group of the index space, long range memory, and often local $H$-self-similarity.
	
	Alas, it does not always exist. Moreover, it is in general not easy to check if it does, as one needs to prove the positive definiteness of the corresponding covariance kernel. This question has been of interest earlier in some special cases: Lévy proved the existence of the Brownian field (corresponding to $H=1/2$) indexed by the Euclidean spaces and the spheres~\cite{levyprocessus,levy1959sphere}. He used direct geometrical constructions, generalised later by Chentsov, Morozova \cite{morozova68}, Lifshits \cite{LifshitsLevy} and Takenaka \cite{takenaka87}. Others authors have tackled this question with harmonic analysis : Molchan extensively studied the existence of the Brownian field indexed by symmetric spaces~\cite{molchan87a}. The works by Gangolli~\cite{gangolli67} on Lévy-Schoenberg kernels and by Faraut and Harzallah~\cite{faraut74} on Hilbertian distances are also strongly connected to the question.
	Authors on fractional Brownian fields indexed by Riemannian manifolds includes Istas \cite{istas2005spherical,istas2006karhunen,istas2011manifold}, Cohen and Lifshits \cite{cohen2012stationary}, who considered hyperbolic spaces and spheres.
	
	 The existence of the $H$-fractional Brownian field indexed by a metric space $(E,d)$  is equivalent to the negative definite property of the kernel $d^{2H}$, where $d$ is the distance on the index space $(E,d)$. Istas \cite{istas2011manifold} noticed that there exists a \emph{fractional index} $\beta_E>0$ depending on $(E,d)$ such that $d^{2H}$ is a negative definite kernel if and only if $2H\leq \beta_E$. It is clear from Faraut and Harzallah \cite{faraut74} that any Hilbert space and thus Euclidean spaces enjoy fractional indexes equal to~$2$, and Istas showed that this is the maximum value of $\beta_E$ for a Riemannian manifold (\cite{istas2011manifold}). However the spheres and the real hyperbolic spaces have fractional index $1$ (see Istas \cite{istas2011manifold}). As $\beta_E<2$ for any Riemannian manifold with at least one point of positive curvature (Istas \cite{istas2011manifold}) and $\beta_E<1$ for any ellipsoid which is not a sphere (Chentsov and Morozova \cite{morozova68}), existence 
	of fractional Brownian fields seems to be a rather fragile property. Results on fractional index related to curvature and topology are given in \cite{feragen} and \cite{Venet_critical}.
	
	Furthermore the negative definiteness of $d^{2H}$ on a metric space also gives the positive definiteness of the kernels of the form $F(d^{2H})$, where $F$ is a completely monotone function (see for instance \cite{Berg_al}). This method gives the existence of stationary Gaussian random fields indexed by the metric space. The associated kernels include the Gaussian and exponential kernels, and are crucial to allow for ``kernel method" machine learning of nonlinear data (see for example \cite{scholkopf2002learning}). \\

	We show in this paper that the cylinder $\mathbb{S}^1 \times ]0,\varepsilon[$ endowed with its canonical Riemannian product metric has fractional index~$0$. In other terms, for every positive $H$ and $\varepsilon$, $d^{2H}_{\mathbb{S}^1 \times ]0,\varepsilon[}$ is not a negative definite kernel, hence there exist no $H$-fractional Brownian field indexed by the cylinder (see Theorem \ref{Thm:cylinder}).
	
	We then generalise this result to the product of two Riemannian manifolds $M\times N$ endowed with the Riemannian product distance, as long as it contains a minimal closed geodesic (see Theorem \ref{Thm:extensionproducts}).
	
	It is remarkable that the cylinder exhibits an obstruction to the negative definiteness of $d^{2H}$ which relies entirely on its global structure:  indeed the cylinder and the Euclidean plane share the same local flat metric, but as one enjoys a negative definite kernel $d^{2H}$ for every $H\leq 1$, the other admits none. To the knowledge of the author the only other known examples of spaces with fractional index $0$ are a class of non Euclidean normed vectorial spaces (see \cite{Koldobskybook}) and the quaternionic hyperbolic space endowed with its geodesic distance (Faraut \cite{faraut74}).
	
	However this global characteristic of the result contrasts with a local aspect. The result is given for a cylinder as short as wanted: the fractional index of $\mathbb{S}^1\times ]0,\varepsilon[$  is zero for every positive $\varepsilon$. 
	We use this ``locallity around the circle" to investigate the case of metric spaces with a distance asymptotically close to the cylinder distance in the neighbourhood of a circle (see Section \ref{sec:perturbation}) and obtain upper bounds for the fractional index in this setting. In particular we look into the case of revolution surfaces and give an example with zero fractional index (see Theorem \ref{Thm:revolution}), which indicates that our argument does not depend on the product structure.
	
	This local feature of our result brings out a discontinuous behaviour of the fractional index, since $\mathbb{S}^1 \times [0,\varepsilon]$ has fractional index $0$ when $\varepsilon$ is positive and $1$ when $\varepsilon=0$. We show in particular that the fractional index is not continuous with respect to the Gromov-Hausdorff convergence of compact metric spaces (see Section \ref{sec:discontinuity}).

	The proof on the cylinder (Theorem \ref{Thm:cylinder}) is done through a direct method by exhibiting for every \mbox{$H<1/2$} a collection of points $(P^H_{i,n})_{i\leq n}$ and of coefficients $(c_i)$ such that \begin{equation} \label{eq:wanted} \lim\limits_{n \rightarrow \infty} \sum_{i,j=1}^n c_ic_j d^{2H}(P^H_{i,n},P^H_{j,n})=+ \infty,\end{equation} which by definition prevents the kernel $d^{2H}$ to be negative definite. We start by investigating a collection of points on the circle, which we afterwards duplicate on two horizontal circles of the cylinder. Finally we consider the same collection of points on a number of circles depending on $n$. The behaviour of  $\lim\limits_{n \rightarrow \infty} \sum_{i,j=1}^n c_ic_j d^{2H}(P^H_{i,n},P^H_{j,n})$ when $n \rightarrow \infty$ is governed by the asymptotic regime of the distance $z_n$ between two consecutive circles, which should be chosen carefully in order to obtain the desired divergence towards infinity.	This adequate regime depends on $H$. In particular every point converges towards the circle at height zero so that the proof works for every $\varepsilon$, but at a rate slow enough so that the quantity we consider in \eqref{eq:wanted} does not asymptotically behave as if the points were on a circle.  Our other results all rely on Theorem \ref{Thm:cylinder}.

	\subparagraph{Outline of the article} In Section \ref{sec:generalities} we recall some generalities and detail our motivations. In Section \ref{sec:Mainstatement} we give the main statement and its proof. In Section \ref{sec:products} we extend our result to Riemannian products. In Section \ref{sec:perturbation} we consider distances close to the product distance on the cylinder. In Section \ref{sec:discontinuity} we deduce that the fractional index is discontinuous with respect to the Gromov-Hausdorff convergence.
	
	\section{Generalities} \label{sec:generalities}
	
	 In this article we consider metric spaces $(E,d)$ and study the negative definiteness property for the functions $d^{2H}(x,y)$, where $H$ is a positive parameter.
	 
	 The metric spaces we consider are cylinders (Section \ref{sec:Mainstatement}) or are close to cylinders in various ways (product spaces with a minimal closed geodesic in Section \ref{sec:products}, spaces which are asymptotically close to a cylinder in Section \ref{sec:perturbation}).
	 
	  In practice, we are looking for the \emph{fractional index} $\beta_E$ of the metric space, which is defined as the supremum of the parameters $H$ such that $d^{2H}$ is negative definite. This index is of particular interest because the function $d^{2H}$ is negative definite if and only if
	 \begin{equation} 2H \leq \beta_E. \end{equation}	 
	 This problematic is motivated by existence problems for fractional Brownian fields and stationary random fields indexed by $(E,d)$, which depend on the negative definiteness of $d^{2H}$. This property also gives the positive definiteness of kernels that are crucial for machine learning of nonlinear data.
	
	In this section we recall some generalities and give details about these motivations.
	
	\paragraph{Positive and negative definite kernels} Given a set $S$, we say that a symmetric function $f:S\times S \rightarrow \mathbb{R}$ is a \emph{positive definite kernel} if for every $x_1,\cdots,x_n \in S$ and every $\lambda_1,\cdots,\lambda_n \in \mathbb{R}$,
		\begin{equation} \label{eq:positive_definite_kernel} \sum_{i,j=1}^n \lambda_i \lambda_j f(x_i,x_j) \geq 0. \end{equation}
		
		Positive definite kernels are the covariances of random fields indexed by $S$. In particular, there exists a centred Gaussian random field indexed by $S$ with covariance $f$ if and only if $f$ is a positive definite kernel (see for instance \cite{Lifshitsbook}). Furthermore they are a key ingredient to machine learning of nonlinear data, as the positive definiteness of $f$ is equivalent to the existence of an Hilbert space $\mathcal{H}$ and a map $\Phi : S \rightarrow \mathcal{H}$ (the ``feature map") such that
		\begin{equation} \label{eq:feature_map} f(x,y)=\langle \Phi(x),\Phi(y) \rangle_{\mathcal{H}}, \end{equation}
 		which guaranties that $f$ can play the role of a scalar product to allow for every linear machine learning method (see \cite{scholkopf2002learning}).
		
		Positive definite kernels are closely related to negative definite kernels (see for example \cite{Berg_al}): a symmetric function $f$ is said to be a \emph{negative definite kernel} if for every $x_1,\cdots,x_n \in S$ and every $c_1,\cdots,c_n \in \mathbb{R}$ such that $\sum_{i=1}^n c_i = 0$,
				\begin{equation} \label{eq:negative_definite_kernel} \sum_{i,j=1}^n c_i c_j f(x_i,x_j) \leq 0. \end{equation}
		
	\paragraph{Fractional Brownian fields}
		Given a metric space $(E,d)$ and $H>0$, we recall that an \emph{$H$-fractional Brownian field} indexed by $E$ is a centred, real-valued, Gaussian random field $(X_x)_{x\in E}$ such that
		\begin{equation} \label{eq:defgen} \forall x,y \in E, ~ \E \left( X_x-X_y\right)^2=\left[d(x,y)\right]^{2H}. \end{equation}
		
		This definition does not yield uniqueness (in law) of the field. Indeed for $N$ a centred Gaussian random variable, if $(X_t)$ is an $H$-fractional Brownian field indexed by $E$ then so is $(N+X_t)$. It is classical to define for any point $O \in E$ the \emph{$H$-fractional Brownian field with origin in $O$} by requiring also that $X_O$ be equal to $0$ almost surely. If it exists one can check that the covariance is then
		\begin{equation} \label{cov} \E(X_xX_y)=\frac{1}{2}\left( d^{2H}(O,x)+d^{2H}(O,y)-d^{2H}(x,y)\right),\end{equation}
		hence the uniqueness of the law of the field. Moreover the existence of the fractional Brownian field with origin in $O$ is equivalent to the positive definiteness of \eqref{cov}. A theorem of Schoenberg (see for example \cite{istas2011manifold}) proves that it is the case if and only if $d^{2H}$ is a negative definite kernel. Notice that this property does not depend on the origin $O$, and that any Gaussian field verifying \eqref{eq:defgen} is obtained by addition of a normal random variable to  an $H$-fractional with origin in an arbitrary $O \in E$: the negative definiteness of $d^{2H}$ is equivalent to the existence of every $H$-fractional Brownian field indexed by $(E,d)$.
		
		\begin{Rem} \label{Rem:stable} In \cite{istas2011manifold} Istas define an $\alpha$-stable $H$-fractional field indexed by a metric space. Unlike in the Gaussian case, positive definiteness of the covariance is not sufficient to guaranty the existence of this field, but it is still necessary that $d^{2H\alpha}$ be negative definite:e: studying the negative definiteness of the powers of $d$ is also a first step for fractional non Gaussian modelling.
		\end{Rem}
	
	\paragraph{Stationary kernels} Furthermore when $d^{2H}$ is negative definite, for every \emph{completely monotone function} $F: \mathbb{R}^+\rightarrow \mathbb{R}^+$,
	\begin{equation} \label{eq:statio_kernels} (x,y) \mapsto F\left(d^{2H}(x,y)\right)\end{equation}
	is a positive definite kernel (see for instance \cite{Berg_al}). Let us recall that a function $F$ is completely monotone if and only $(-1)^n F^{(n)}(t) \geq 0$ for every $t\in \mathbb{R}^+$ and $n\in \mathbb{N}$.
	Since the kernels in \eqref{eq:statio_kernels} depend only on the distance, they are the covariances of \emph{stationary} Gaussian random fields. These are first-choice random models for functions over $E$, whose random behaviour is homogeneous with respect to the geometry of $(E,d)$.
	
	Positive definite kernels that are functions of a distance are also of crucial importance in kernel machine learning, since by replacing a scalar product in learning methods a kernel plays the role of a ``proximity measure". Examples of completely monotone functions include $t\mapsto e^{-\lambda t}$ for every positive $\lambda$. In particular, when they exist $e^{-\lambda d(x,y)}$ and $e^{-\lambda d^2(x,y)}$ generalise the exponential and the Gaussian kernel families.
	
	\paragraph{Fractional index} It is a striking fact that for every metric space $(E,d)$ there exists $\beta_E$ in $[0,+\infty]
	$ such that for every positive $H$,  $d^{2H}$ is negative definite if and only if (see Istas \cite{istas2011manifold})
	
	\begin{equation}\label{eq:fractionalindex} 2H\leq \beta_E.\end{equation}
	
	The number $\beta_E$ is called the fractional index of $(E,d)$ and is in general not easy to compute.  Let us stress out some general facts which follow directly from the definition of $\beta_E$, and that we will use later.
	
	\begin{Rem} \label{Rem:subspace} Given a metric space $(E,d)$ and $F \subset E$, if we consider $F$ as a metric space endowed with the restriction $d_{|F}$ of the distance $d$ to $F$, we have $\beta_F \geq \beta_E$.
	\end{Rem}
	\begin{Rem} \label{Rem:homo} For a positive $\lambda$, multiplying the distance on $E$ by $\lambda$ does not change the fractional index $\beta_E$.
	\end{Rem}
	
    \paragraph{General assumptions on Riemannian manifolds} Most of the time we will consider as index space a Riemannian manifold $M$ endowed with the geodesic distance $d_M$ associated to its Riemannian metric $\langle ~,~ \rangle_M$. Following \cite{gallot} we consider only $C^\infty$, connected, and countable at infinity manifolds in this whole document. Furthermore we assume the manifolds to be connected and without boundary, with the notable exception of $\mathbb{S}^1 \times [0,\varepsilon]$ in the proof of Theorem \ref{Thm:Gromovdiscontinuity}.
	
	\begin{Rem} \label{Rem:submanifold} Given a Riemannian manifold $M$ and a submanifold $N$ of $M$, it is possible to consider the restriction $d_{M|N}$ of the geodesic distance $d_M$ to $N$. On the other hand, one can consider the Riemannian manifold $N$ endowed with the restriction $\langle ~,~ \rangle_{M|N}$ of the inner product of $M$ to $N$, which gives a geodesic distance $d_N$. In general those two distances are not equal, because the minimal geodesics in $M$ from points of $N$ take values in the whole of $M$. In particular it is not possible to deduce the value of $\beta_M$ from local aspects of $M$ only, in spite of Remark \ref{Rem:subspace}.\end{Rem}
	\paragraph{Minimal closed geodesics} Let us recall that a \emph{minimal closed geodesic} $\gamma$ in a Riemannian manifold $M$ is a closed curve with values in $M$ such that for every points $P,Q$ on $\gamma$ there exists a minimal geodesic joining $P$ to $Q$ that is included in $\gamma$. In this case the two distances $d_{M|\gamma}$ and $d_{\gamma}$ are equal, and $\gamma$ is isometric to a circle of length $L(\gamma)$. In particular for a Riemannian manifold with a minimal closed geodesic, $\beta_M \leq \beta_{\mathbb{S}^1}=1$ (see Remarks \ref{Rem:subspace} and \ref{Rem:homo}).

	\section{Main result} \label{sec:Mainstatement}
	
	In this section we consider the cylinder $\mathbb{S}^1\times \mathbb{R}$ endowed with its Riemannian product metric
	\begin{equation}\label{eq:cylinder_metric} \langle~,~ \rangle_{\mathbb{S}^1 \times \mathbb{R}}=d\theta^2+dz^2. \end{equation}
	The expression of the associated geodesic distance is
	\begin{equation} d_{\mathbb{S}^1 \times \mathbb{R}}((\theta_1,z_1),(\theta_2,z_2))=\left(d_{\mathbb{S}^1}(\theta_1,\theta_2)^2+|z_1-z_2|^2\right)^{1/2},\end{equation}
	where $d_{\mathbb{S}^1}$ is the geodesic distance on $\mathbb{S}^1$, given by
	\begin{equation} d_{\mathbb{S}^1}(\theta_1,\theta_2)=\min(|\theta_1-\theta_2|,2\pi-|\theta_1-\theta_2|).\end{equation}
	
	\begin{Rem} \label{Rem:coincides} In the cylinder the geodesics are given by arcs of helices. In particular all the geodesics in $\mathbb{S}^1 \times \mathbb{R}$ between points of $\mathbb{S}^1 \times ]0,\varepsilon[$ stay at all time in $\mathbb{S}^1 \times ]0,\varepsilon[$. As a consequence, the restriction of $d_{\mathbb{S}^1 \times \mathbb{R}}$ to $\mathbb{S}^1 \times ]0,\varepsilon[$ and the geodesic distance associated to the metric \eqref{eq:cylinder_metric} on $\mathbb{S}^1\times ]0,\varepsilon[$ coincide.\end{Rem}

	\begin{Thm} \label{Thm:cylinder}
		For every $\varepsilon>0$ and $H >0$, $d^{2H}_{\mathbb{S}^1\times ]0,\varepsilon[}$ is not negative definite, hence there exists no $H$-fractional Brownian field indexed by the cylinder $\mathbb{S}^1 \times ]0,\varepsilon[$. In other terms, $$\beta_{\mathbb{S}^1 \times ]0,\varepsilon[}=0.$$
	\end{Thm} 
 \subsection{Proof of the main result}
 
 Let us give an outline of the proof of the theorem. To prove the result we exhibit for every $0<H<1/2$ a sequence of configurations $$((P^H_{1,n},\cdots,P^H_{n,n}),(c_1,\cdots,c_n))_{n\in\mathbb{N}}$$ such that
 $$\lim\limits_{n \rightarrow \infty} \sum_{i,j=1}^n c_ic_j d_{\mathbb{S}^1\times ]0,\varepsilon[}^{2H}(P^H_{i,n},P^H_{j,n})=+ \infty,$$
 which shows that $d_{\mathbb{S}^1\times ]0,\varepsilon[}^{2H}$ is not a negative definite kernel. Hence there exists no  $H$-fractional \linebreak Brownian field indexed by $\mathbb{S}^1 \times ]0,\varepsilon[$ for every $0<H<1/2$.
 To conclude for every $H>0$ we recall that if $d^{2H}$ is not negative definite then $d^{2H'}$ is not negative definite for every $H' \geq H$ (see \eqref{eq:fractionalindex}).
 
 We carry the proof with a cylinder of radius $\frac{1}{2\pi}$ in order to get parallel circles of perimeter $1$ and lighten the computations. Doing so only multiplies the distance $d_{\mathbb{S}^1 \times \mathbb{R}}$ by a positive constant, therefore the fractional index remains the same (see Remark \ref{Rem:homo}).
 
	In Section \ref{subsec:circle} we work on a sequence of configurations with points on one circle. Section \ref{subsec:duplicate} deals with the same sequence duplicated on two parallel circles of the cylinder. We finish the proof in Section \ref{subsec:prooftheorem} by considering the same sequence of configurations on a diverging number of parallel circles of the cylinder.
	\subsubsection{A configuration on the circle} \label{subsec:circle}

	Let us consider a circle $S$ of \mbox{perimeter $1$}, parametrised by arc length $s \in [0,1[$. In this chart we have an explicit formula for the geodesic distance, $$d_{S}(s,s')=\min(|s-s'|,1-|s-s'|).$$
	For every $N\in\mathbb{N}$ and $1 \leq i \leq 4N$  we define $$P_{i,N}:=\frac{i}{4N} \in S,$$ and the coefficients  
	$$c_i=(-1)^i.$$	 
	Notice that for every $N$ we have $$\sum_{i=1}^{4N} c_i =0,$$
	so that $((P_{1,N}, \cdots P_{4N,N}),(c_1,\cdots, c_{4N}))$ is a configuration of $4N$ points in $S$.
	
	We now deal with the asymptotic behaviour of 
	\begin{equation} \label{A_N} A_N := \sum_{i,j=1}^{4N} c_i c_j d_{S}^{2H}(P_{i,N},P_{j,N}).\end{equation}
	\begin{Lem} \label{Lemcercle}
		For every $H\in ]0,1/2[,$ 
		$$
		\label{eq:AN} A_N \underset{N \rightarrow \infty}{\sim} \frac{N^{1-2H}}{4^{2H-1}} \sum_{p=0}^{\infty} \left[(2p)^{2H}-2(2p+1)^{2H}+(2p+2)^{2H}\right].
		$$
	\end{Lem}
	\begin{proof}
		We write $P_i$ instead of $P_{i,N}$ when there is no ambiguity. The terms $d_{S}(P_{i},P_{j})$ appearing in the sum $A_N$ are of the form $\frac{k}{4N}$ for $k \in \{1, \cdots, 2N\}$. Each one appears $8N$ times except the term for $k=2N$. This last terms only appears $4N$ times corresponding to pairs of antipodal points.
		
		Moreover $c_i c_j$ depends only on $d_S(P_{i},P_{j})$, therefore
		
		\begin{alignat*}{1}
			A_N &=8N\sum_{k=1}^{2N-1} (-1)^k \left(\frac{k}{4N}\right)^{2H}+4N\left(\frac{1}{2}\right)^{2H}\\
			&=8N \left(\sum_{p=1}^{N-1} \left(\frac{2p}{4N}\right)^{2H}- \sum_{p=0}^{N-1}\left(\frac{2p+1}{4N}\right)^{2H}\right)+4N\left(\frac{1}{2}\right)^{2H}\\
			&=4N\left(\sum_{p=1}^{N-1}\left(\frac{2p}{4N}\right)^{2H}-2\sum_{p=0}^{N-1}\left(\frac{2p+1}{4N}\right)^{2H}+\sum_{p=0}^{N-2}\left(\frac{2p+2}{4N}\right)^{2H}\right)+4N\left(\frac{1}{2}\right)^{2H}\\
			&=4N \sum_{p=0}^{N-1} \left[\left(\frac{2p}{4N}\right)^{2H}-2\left(\frac{2p+1}{4N}\right)^{2H}+\left(\frac{2p+2}{4N}\right)^{2H}\right]\\
			&=\frac{4N^{1-2H}}{4^{2H}} \sum_{p=0}^{N-1} \left[(2p)^{2H}-2(2p+1)^{2H}+(2p+2)^{2H}\right].
		\end{alignat*}
		
		Because $$(2p)^{2H}-2(2p+1)^{2H}+(2p+2)^{2H}=O \left(\frac{1}{p^{2-2H}}\right)$$ and $H<1/2$, the series above converge and we get the result.
	\end{proof}
	\begin{Rem}For $H<1/2$ the sum of the series appearing in \eqref{eq:AN} is nonpositive by concavity of $x \mapsto x^{2H}$, hence $\lim\limits_{N\rightarrow \infty} A_N = -\infty$. Because $\beta_S=\beta_{\mathbb{S}^1}=1$, it is clear that no choice of configuration on the circle will give a positive result. It is then necessary to consider points at different heights on the cylinder in order to obtain our result. We start by duplicating our configuration on two circles.
	\end{Rem}
	\subsubsection{Duplicating the circle configuration} \label{subsec:duplicate}
	We now turn to the cylinder $S \times \mathbb{R}$, considering again a circle $S$ of perimeter $1$ parametrised by arc length. In the entire proof of Theorem \ref{Thm:cylinder} we denote by $d$ the geodesic distance $d_{S\times\mathbb{R}}$. Given two points $(s_1,z_1),(s_2,z_2) \in S\times \mathbb{R}$ we have
	$$ d((s_1,z_1),(s_2,z_2))=\left(d_{S}(s_1,s_2)^2+|z_1-z_2|^2\right)^{1/2}.$$
	
	Let us now consider a sequence of positive numbers $(z_N)_{N\in \mathbb{N}}$, and for every $N \in \mathbb{N}$,
	
	$$ P_{i,N}:= \left \{ \begin{array}{ll}\left(\frac{i}{4N},0\right) & \text{if }1 \leq i \leq 4N, \\  \left(\frac{i}{4N},z_N\right) & \text{if } 4N+1 \leq i \leq 8N. \end{array}\right.$$

	We set for every $1\leq i \leq 8N$ $$c_i=(-1)^i,$$ and notice again that $$ \forall N \in \mathbb{N}, ~ \sum_{i=1}^{8N}c_i=0,$$ so that $((P_{1,N}, \cdots P_{8N,N}),(c_1,\cdots, c_{8N}))$ is a configuration of $8N$ points in $S \times \mathbb{R}$.
	
	This time we deal with the asymptotic behaviour of 
	\begin{equation} C_N := \sum_{i,j=1}^{8N} c_i c_j d^{2H}(P_{i,N},P_{j,N}). \end{equation}
	We write again $P_i$ instead of $P_{i,N}$ when there is no ambiguity. Let us split
	
	\begin{multline*}
		C_N = \sum_{i,j=1}^{4N} (-1)^{i+j} [d(P_i,P_j)]^{2H}+\sum_{i,j=4N+1}^{8N} (-1)^{i+j} [d(P_i,P_j)]^{2H} \\ + \sum_{i=1}^{4N} \sum_{j=4N+1}^{8N} (-1)^{i+j} [d(P_i,P_j)]^{2H} + \sum_{i=4N+1}^{8N} \sum_{j=1}^{4N} (-1)^{i+j} [d(P_i,P_j)]^{2H}.
	\end{multline*}
	
	We now write $$ C_N=2A_N + 2B_N(z_N),$$ with $A_N$ as in \eqref{A_N} and 
	\begin{equation}B_N(z_N) := \sum_{i=1}^{4N} \sum_{j=4N+1}^{8N} (-1)^{i+j} [d(P_i,P_j)]^{2H}. \end{equation}								
	Since we know from Lemma \ref{Lemcercle} how $A_N$ behaves it remains to work on $B_N$ under proper assumptions on the regime $z_N$. Because $A_N$ is non positive, we aim to get a positive contribution from $B_N$. Asymptotic order of $B_N$ is also crucial in order to outweigh $A_N$, which we have proven to have asymptotic order $N^{1-2H}$. From our investigations it seems that
	\begin{itemize}
		\item if $z_N$ converges too quickly to zero $B_N$ tends to behave like $A_N$. In particular setting $$z_N=\frac{z_0}{N}$$ yields $$B_N \underset{N \rightarrow \infty}{\sim} C(z)N^{1-2H},$$ with $C(z_0)$ continuous in $z_0$. Since setting $z_0=0$ gives $B_N=A_N$, it is clear that $C(z_0)$ is non positive for small values of $z_0$, which is problematic because we aim at considering cylinders of the form $S \times ]0,\varepsilon[$ with $\varepsilon$ as small as desired.
		\item Choosing $z_N$ with slower regimes yields positive contribution from $B_N$ at the expense of a less important asymptotic order. In particular setting $$z_N=z_0>0$$ yields $$B_N \underset{N \rightarrow \infty}{\longrightarrow} \frac{H}{2}\left(\frac{1}{4}+z_0^2\right)^{H-1}$$ which is negligeable in front of $|A_N|$. 
	\end{itemize}
	
	We now give a class of regimes for $z_N$ under which $B_N(z_N)$ converges to a positive constant independent of $z_N$, with uniform speed in $z_N$.  We will later take advantage of this fact to consider an infinite number of circles and recover a dominant asymptotic order for $B_N(z_N)$.
	\begin{Lem}\label{lemdupli}
		Let us denote by $\mathcal{Z}_{\underline{\alpha},\overline{\alpha}}$ the set of all sequences of positive numbers $\left(z_N\right)_{N \geq 0}$ such that
		\begin{equation} \label{H_1} \tag{H1} z_N N^{\underline{\alpha}} \underset{N \rightarrow \infty}{\longrightarrow} 0 \end{equation}
		and
		\begin{equation} \label{H_2} \tag{H2} z_N N^{\overline{\alpha}} \underset{N \longrightarrow \infty}{\longrightarrow} \infty. \end{equation}
		For every $0<H<1/2$ and $\underline{\alpha},\overline{\alpha}$ such that $ 0 < \underline{\alpha} < \overline{\alpha} < 1$ we have
		
		$$\lim\limits_{N\rightarrow \infty} \sup\limits_{(z_N)_{N \geq 0} \in \mathcal{Z}_{\underline{\alpha},\overline{\alpha}}} \left| B_N(z_N)-\frac{H}{2\cdot 4^{H-1}}\right|=0.$$
	\end{Lem}
	
	\subparagraph{Notations}
	
	We introduce some notations we use in the whole proof of Lemma \ref{lemdupli}. Let us write \begin{equation} \label{Notations}
		\left. \begin{array}{c}
			\alpha_N=\frac{-\ln(z_N)}{\ln(N)}, \\ \\ ~ \varphi : x \mapsto (x^2+1)^H, \\ \\ ~ x_p=\frac{2p+1}{4N^{1-\alpha_N}},\\ \\ ~ h=\frac{1}{4N^{1-\alpha_N}},  \\ \\ \theta_l = \alpha_N (l-1-2H)-l+2. \end{array} \right\} \end{equation}
	
	Because we aim for a result with uniformity in $z_N$, from now on we denote by
	\begin{itemize}
		\item $a(N,z_N)=O_u(b(N,z_N))$ the existence of $ C > O$ and $N_0$ such that for every $z_N \in \mathcal{Z}_{\underline{\alpha},\overline{\alpha}}$ and $N \geq N_0$, $|a(N,z_N)| \leq C |b(N,z_N)|$.
		\item In a similar way, $a(N,z_N)=o_u(b(N,z_N))$ means that $\forall \varepsilon>0, ~ \exists N_0$, $\forall z_N \in \mathcal{Z}_{\underline{\alpha},\overline{\alpha}}$, $|a(N,z_N)|\leq \varepsilon|b(N,z_N)|$.
	\end{itemize}
	
	To prove Lemma \ref{lemdupli} we proceed through Lemma \ref{Lemdupli2}, Lemma \ref{Lemdupli3}, and Lemma \ref{Lemdupli4} to a Taylor-like expansion of $B_N(z_N)$ on the powers $N^{\theta_l}$. Observe that for every $l$ we have $N^{\theta_{l+1}}=o\left( N^{\theta_l}\right).$ Indeed
	
	$$\frac{N^{\theta_{l+1}}}{N^{\theta_{l}}}=N^{\alpha_N-1}=\frac{1}{z_N N}=\frac{1}{z_N N^{\overline{\alpha}}} \times \frac{N^{\overline{\alpha}}}{N}$$
	converges towards zero when $N$ goes to infinity (use \eqref{H_2} and $\overline{\alpha}<1)$.
	
	Let us now give Lemma \ref{Lemdupli1} which we will use to estimate the asymptotic order of some remainders in the expansion.
	
	\begin{Lem} \label{Lemdupli1} With the notations from \eqref{Notations}, for every $H<1/2$, every integer $q \geq 2 $ and  $$ y_p =x_p+ h  \delta_{p,N}, $$ where $\delta_{p,N}$ is any double-indexed sequence with values in $\left[-1,1\right],$												
		$$\sum_{p=0}^{N-1}|\varphi^{(q)}(y_p)|=O_u\left(N^{1-\alpha_N}\right).$$
	\end{Lem}
	\begin{proof}
		Along the proof we use the positive constants $C_1,\cdots,C_6$. We claim that they exist and are independent of $N$ and the choice of $z_N \in \mathcal{Z}_{\underline{\alpha},\overline{\alpha}}$, though some may depend in $q$ and $H$ without altering the result. Let us notice that $$\varphi^{(q)}(t) \underset{t \rightarrow \infty}{\sim} C_1~ t^{2H-q},$$ which yields $$\varphi^{(q)}(t) \leq C_2 ~ t^{2H-q}.$$ We obtain
		\begin{alignat*}{1}
			\sum_{p=0}^{N-1}|\varphi^{(q)}(y_p)| &\leq   \sum_{p=0}^{\lfloor N^{1-\alpha_N}\rfloor} ||\varphi^{(q)}||_\infty + ~~ C_2 \! \! \! \! \!  \! \! \!\sum_{p={\lfloor N^{1-\alpha_N}\rfloor}+1}^{N-1} (y_p)^{2H-q},
		\end{alignat*}
		and $$(y_p)^{2H-q}=x_p^{2H-q}  \left( 1+ \frac{h \delta_{p,N} }{x_p} \right)^{2H-q}\leq C_3 ~ x_p^{2H-q}$$  because 
		$$\left( 1+ \frac{h \delta_{p,N}}{x_p} \right)=\left( 1+\frac{\delta_{p,N} }{2p+1} \right)$$  is bounded and away from $0$ as long as $p>0$.
		
		Finally
		\begin{alignat*}{1}
			\sum_{p=0}^{N-1}|\varphi^{(q)}(\delta_p)|&\leq (\lfloor N^{1-\alpha_N} \rfloor +1 ) ||\varphi^{(q)}||_\infty  + ~ C_2 C_3 \! \! \! \! \sum_{p={\lfloor N^{1-\alpha_N}\rfloor}+1}^{N-1} \left(\frac{2p+1}{4N^{1-\alpha_N}}\right)^{2H-q}\\
			&\leq C_4  N^{1-\alpha_N} + C_5 ~\frac{1}{\left(N^{1-\alpha_N}\right)^{2H-q}}  \sum_{p={\lfloor N^{1-\alpha_N}\rfloor}+1}^{N-1} p^{2H-q}\\
			&\leq C_4  N^{1-\alpha_N} + C_6 ~\frac{1}{\left(N^{1-\alpha_N}\right)^{2H-q}}  \left(N^{1-\alpha_N}\right)^{2H-q+1}\\
			&= O_u\left(N^{1-\alpha_N}\right). \qedhere
		\end{alignat*}
	\end{proof}
	\begin{Lem}\label{Lemdupli2} Under the assumptions \eqref{H_1}, \eqref{H_2} and with the notations \eqref{Notations} we have for every $M \geq 2$
		$$B_N(z_N)  =\sum_{n=2}^M b_n B_N^n+O_u\left(N^{\theta_{M+1}}\right), \label{firstexpansion}$$
		with
		\begin{equation}B_N^n  :=N^{\theta_n}\sum_{p=0}^{N-1} \frac{1}{2 N^{1-\alpha_N}} \varphi^{(n)}(x_p) \end{equation} and \begin{equation} b_n:=\frac{8}{n! 4^n}(1+(-1)^n).\end{equation}
	\end{Lem}
	\begin{proof} We start by reordering the terms in $B_N(z_N)$ in a similar way as we did for $A_N$ in the proof of Lemma \ref{Lemcercle}:
		
		\begin{alignat*}{1}
			& B_N(z_N)= 4N z_N^{2H} + 8N\sum_{k=1}^{2N-1} (-1)^k \left[\left(\frac{k}{4N}\right)^2+z_N^2\right]^{H} + 4N\left[\frac{1}{2^2}+z_N^2\right]^{H}\\
			& = 4N\sum_{p=0}^{N-1}\left(\left[\left(\frac{2p}{4N}\right)^2+z_N^2\right]^H-2 \left[\left(\frac{2p+1}{4N}\right)^2+z_N^2\right]^H+\left[\left(\frac{2p+2}{4N}\right)^2+z_N^2\right]^H\right)\\
			& = 4N \sum_{p=0}^{N-1}\left( \left[\left(\frac{2p}{4N}\right)^2+\frac{1}{N^{2\alpha_N}}\right]^H \! \! \! \! - 2 \left[\left(\frac{2p+1}{4N}\right)^2+\frac{1}{N^{2\alpha_N}}\right]^H \! \! \! \! + \left[\left(\frac{2p+2}{4N}\right)^2+\frac{1}{N^{2\alpha_N}}\right]^H  \right)  \\
			& =  \frac{4N}{N^{2\alpha_NH}} \sum_{p=0}^{N-1}\left( \left[\left(\frac{2p+1-1}{4N^{1-\alpha_N}}\right)^2 \! \! +1\right]^H \! \! \! \! - 2 \left[\left(\frac{2p+1}{4N^{1-\alpha_N}}\right)^2 \!  \! + 1\right]^H \! \! \! \!  + \left[\left(\frac{2p+1+1}{4N^{1-\alpha_N}}\right)^2 \! \! +1\right]^H  \right)\\
			& =  4N^{1-2\alpha_N H} \sum_{p=0}^{N-1} \left[ \varphi \left(x_p-h\right)-2\varphi \left(x_p\right)+\varphi \left(x_p+h\right) \right],
		\end{alignat*}
		
		Taylor expansions of $\varphi$ up to an arbitrary order $M$ give the following approximation of $B_N(z_N)$:
		\begin{alignat*}{1}
			& 4N^{1-2\alpha_N H} \sum_{p=0}^{N-1} \sum_{n=2}^M \left[\left(-h\right)^n\frac{\varphi^{(n)}(x_p)}{n!}+h^n\frac{\varphi^{(n)}(x_p)}{n!}\right]\\
			& = N^{1-2\alpha_N H} \sum_{p=0}^{N-1} \sum_{n=2}^M \frac{b_n}{2\left(N^{1-\alpha_N}\right)^n} \varphi^{(n)}(x_p) \\
			& =\sum_{n=2}^M b_n ~ N^{\theta_n}\sum_{p=0}^{N-1} \frac{1}{2 N^{1-\alpha_N}} \varphi^{(n)}(x_p), \text{ with the remainder}
		\end{alignat*}
		\begin{equation*}\label{R_{M+1}}R_{M+1}:= N^{1-2\alpha_N H} \sum_{p=0}^{N-1} \frac{C_M}{N^{(1-\alpha_N)(M+1)}} \left[\varphi^{(M+1)}(y_{p,1})+ (-1)^{(M+1)}\varphi^{(M+1)}(y_{p,2}) \right],
		\end{equation*}
		where $$y_{p,1} \in  ]x_p-h,x_p[$$ and  $$y_{p,2} \in ]x_p,x_p+h[ \ .$$
		
		Using Lemma \ref{Lemdupli1} with $y_{p}=y_{p,1}$ and again with $y_{p}=y_{p,2}$ shows that $$R_{M+1}= O_u \! \left(N^{\theta_{M+1}}\right) \ .\qedhere$$\end{proof}
	
	\begin{Lem} \label{Lemdupli3} Under the assumptions \eqref{H_1}, \eqref{H_2} and with the notations \eqref{Notations},
		for every $n \geq 3$ and $M\geq n:$
		\begin{equation}
			\label{BNn} B_N^n= \sum_{k=0}^{M-n} d_k N^{\theta_{n+k}} \varphi^{(n+k-1)}(0)+ \sum_{k=1}^{M-n} a_k ~ B_N^{n+k}+ O_u\left(N^{\theta_{M+1}}\right)+ o_u \left(1\right),
		\end{equation}
		while for every $M\geq2:$									\begin{equation}
			\label{BN2} B_N^2= \frac{H}{4^{H-1}} + \sum_{k=0}^{M-2} d_k N^{\theta_{2+k}} \varphi^{(2+k-1)}(0)+ \sum_{k=1}^{M-2} a_k ~ B_N^{2+k}+ O_u\left(N^{\theta_{M+1}}\right) + o_u \left(1\right),
		\end{equation}
		\begin{equation}\text{with~~~ }d_k:=-\frac{1}{4^k k!} \ , \end{equation} \begin{equation}a_k:=-\frac{1}{2^k (k+1)!}  \ .\end{equation}
	\end{Lem}
	\begin{proof}Let us write \begin{align*}
			B_N^n  =N^{\theta_n}\sum_{p=0}^{N-1} \frac{1}{2 N^{1-\alpha_N}} \varphi^{(n)}(x_p) = N^{\theta_n} \sum_{p=0}^{N-1} \int_{x_p}^{x_{p+1}} \varphi^{(n)} (x_p) dt.
		\end{align*}
		Proceeding to a Taylor expansion up to the order $M-n$ of $\varphi^{(n)}(t)$ for any $t \in [x_p, x_{p+1}]$, we write, calling $R^n_{M+1}$ the remainder from the Taylor expansion:
		\begin{multline} \label{BN} B_N^n = N^{\theta_n} \left(\sum_{p=0}^{N-1} \int_{x_p}^{x_{p+1}} \left[\varphi^{(n)}(t) - \sum_{k=1}^{M-n} \frac{(t-x_p)^k}{k!}\varphi^{(n+k)}(x_p) \right] dt\right) + R^n_{M+1} \\
			=  N^{\theta_n} \left(\left[\varphi^{(n-1)}(t)\right]_{x_0}^{x_N} \! \!  - \! \! \sum_{k=1}^{M-n} \frac{1}{2^k (k+1)!} \cdot \frac{1}{N^{(1-\alpha_N)k}}\sum_{p=0}^{N-1} \frac{1}{2N^{1-\alpha_N}} \varphi^{(n+k)}(x_p)\right) \\ + R^n_{M+1} \ . \end{multline}
		For every $p$ and $t\in[x_p,x_{p+1}]$, there exists $y_p(t)$ in $]x_p, x_{p+1}[$ and continuous in $t$ such that 
		$$ R^n_{M+1} = N^{\theta_n} \left( - \sum_{p=0}^{N-1} \int_{x_p}^{x_{p+1}} \frac{(t-x_p)^{M-n+1}}{(M-n+1)!} \varphi^{(M+1)} (y_p(t)) dt \right) \ .$$
		
		We have
		\begin{alignat*}{1}
			\left|R^n_{M+1}\right| \leq ~ &   N^{\theta_n} \sum_{p=0}^{N-1} \max\limits_{t\in[x_p,x_{p+1}]} \left| \varphi^{(M+1)} (y_p(t))\right| \int_{x_p}^{x_{p+1}} \frac{(t-x_p)^{M-n+1}}{(M-n+1)!} dt\\
			=~ & N^{\theta_n} \sum_{p=0}^{N-1} \left| \varphi^{(M+1)} (y'_p)\right| ~ O_u\left(N^{(\alpha_N-1)(M-n+2)}\right)\\
			& \text{for some } y'_p \in \argmax\limits_{t\in[x_p,x_{p+1}]} \left| \varphi^{(M+1)} (y_p(t))\right|
		\end{alignat*}
		Using Lemma \ref{Lemdupli1} again we obtain \begin{equation} \label{reste} \left|R^n_{M+1}\right| = O_u\left(N^{\theta_{M+1}}\right). \end{equation}
		Coming back to \eqref{BN}, \begin{itemize} \item for $n=2$ it is easy to see that  \begin{alignat*}{1}\label{piece1}
				N^{\theta_n} \varphi^{(n-1)}(x_N) & = N^{\theta_2} \varphi'(x_N)= \frac{H}{4^{H-1}}+O_u \left( z_N^2\right) + o_u \left( 1 \right).\end{alignat*}
			Using \eqref{H_1} we obtain \begin{equation} \label{n3} N^{\theta_n} \varphi^{(n-1)}(x_N)=\frac{H}{4^{H-1}} + o_u(1) \end{equation} \item while for  $n \geq 3$ \begin{equation}\label{n2}N^{\theta_n} \varphi^{(n-1)}(x_N)=o_u \left(1\right).\end{equation}
		\end{itemize}
		In both cases, expanding $\varphi^{(n-1)}$ up to the order $M-n$ and using Lemma \ref{Lemdupli1} again to deal with the remainder we get
		\begin{equation}\label{piece3}-N^{\theta_n} \varphi^{(n-1)}(x_0)=\sum_{k=0}^{M-n} -\frac{N^{\theta_{n+k}}}{4^k k!} \varphi^{(n+k-1)}(0) + O_u \left(N^{\theta_{M+1}} \right).\end{equation}
		It remains in \eqref{BN} the term
		\begin{alignat}{1}\label{piece4}
			& \notag N^{\theta_N}\left(- \sum_{k=1}^{M-n} \frac{1}{2^k (k+1)!} \cdot \frac{1}{N^{(1-\alpha_N)k}}\sum_{p=0}^{N-1} \frac{1}{2N^{1-\alpha_N}} \varphi^{(n+k)}(x_p)\right)\\ = & \sum_{k=1}^{M-n} - \frac{1}{2^k (k+1)!} B_N^{n+k}.
		\end{alignat}
		Putting together all the pieces of \eqref{BN} from \eqref{reste}, \eqref{piece3}, \eqref{piece4}, and \eqref{n3} or \eqref{n2} whether $n=2$ or $n\geq 3$, we get the result.\end{proof}
	\begin{Lem}\label{Lemdupli4}Under the assumptions \eqref{H_1}, \eqref{H_2} and with the notations \eqref{Notations}, for every $M'\geq 2$ we have
		
		$$B_N(z_N)= \frac{H}{2 \cdot 4^{H-1}} + \sum_{l'=1}^{M'-1} C_{2l'+1} N^{\theta_{2l'+1}} \varphi^{2l'}(0)+ O_u \left( N^{\theta_{2M'+1}}\right)+ o_u\left(1\right), $$
		with
		\begin{equation} C_l:= \sum_{n'=1}^{\lfloor l/2 \rfloor} \sum_{k=0}^{l-2n'} b_{2n'} ~ A_{l-2n'-k} ~ d_k,\end{equation} where $ A_0:=1$ and for every $p\geq 1,$ \begin{equation} A_p:= \sum_{q=1}^{p} \!\!\!\!\!\! \sum_{\substack{~~~~m_1,\cdots,m_q >0 \\ ~~~~m_1+\cdots+m_q=p }} \! \! \!\!\!\!\!\!\!\!\!  a_{m_1}\cdots a_{m_q}.\end{equation}
	\end{Lem}
	\begin{proof} 
		Using \eqref{BNn} and \eqref{BN2} from Lemma \ref{Lemdupli3} we get
		
		\begin{multline*}
			B_N(z_N)=  b_2  \Biggl(  \frac{H}{4^{H-1}}+ \sum_{k=0}^{M-2} d_k N^{\theta_{2+k}} \varphi^{(2+k-1)}(0) + \sum_{k=1}^{M-2} a_k ~ B_N^{2+k} \Biggr) \\ + \sum_{n=3}^M b_n \left( \sum_{k=0}^{M-n} d_k N^{\theta_{n+k}} \varphi^{(n+k-1)}(0)+ \sum_{k=1}^{M-n} a_k ~ B_N^{n+k}\right) + O_u\left(N^{\theta_{M+1}}\right)+o_u\left(1\right),
		\end{multline*}
		
		gathering terms and using $b_2=\frac{1}{2}$ we obtain
		
		\begin{multline*}
			B_N(z_N)= \frac{H}{2\cdot 4^{H-1}}  + \sum_{n=2}^M b_n \left(\sum_{k=0}^{M-n} d_k N^{\theta_{n+k}} \varphi^{(n+k-1)}(0)+ \sum_{m=1}^{M-n} a_m ~ B_N^{n+m} \right) \\+ O_u\left(N^{\theta_{M+1}}\right) +  o_u\left(1\right).
		\end{multline*}
		
		We now recursively apply \eqref{BNn} to obtain an explicit expansion of $B_N(z_N)$: all the asymptotic terms of the form  $O_u\left(N^{\theta_{M+1}}\right)$ and $o_u\left(1\right)$ gather because we only use \eqref{BNn}  a finite number of times. Apart from $ \frac{H}{2 \cdot 4^{H-1}}$, we only obtain terms of the form $C N^{\theta_l} \varphi^{l-1}(0)$. Furthermore :
		\begin{enumerate}
			\item if the term was obtained without using \eqref{BNn}, $C=b_nd_k$ for some $n$ and $k$ such that $n+k=l$,
			\item if the term was obtained after using \eqref{BNn} $q$ times, $C=b_n a_{m_1}\cdots a_{m_q}d_k$ with $n+m_1+ \cdots + m_q + k = l$.
		\end{enumerate}
		Hence the total constant before $N^{\theta_l} \varphi^{l-1}(0)$ equals
		$$\displaystyle C_l= \sum_{n=2}^l \sum_{k=0}^{l-n} b_n ~ A_{l-n-k} ~ d_k.$$ Let us notice that $b_n = 0$ for odd $n$ and $\varphi^{(l-1)}(0)=0$ for even $l$. We therefore write $n=2n'$, $l=2l'+1$ and $M'=\lceil M/2 \rceil$ and obtain the result.  \end{proof}

	\begin{proof}[Proof of Lemma \ref{lemdupli}] We will now show that all the coefficients $C_{2l'+1}$ in Lemma \ref{Lemdupli4} are vanishing. Let us write
		
		$$\displaystyle{C_{l}=\sum_{n'=1}^{\lfloor l/2 \rfloor} b_{2n'} Z_{l-2n'}}$$
		
		with
		
		$$\displaystyle{Z_r=\sum_{k=0}^r A_{r-k}d_k}$$
		
		for every $r \geq 1$.
		We are going to prove that $Z_r=0$ when $r$ is odd, which implies that $C_l=0$ when $l$ is odd. We do so by finding a formal power series associated to $(Z_r)_{r \geq 1}$ and showing that it converges to an even function.
		
		\begin{alignat*}{1}
			Z_r= \sum_{k=0}^{r-1} A_{r-k}d_k+A_0d_r=\sum_{k=0}^{r-1}\left( \sum_{q=1}^{r-k} \!\!\!\!\!\! \sum_{\substack{~~~~m_1,\cdots,m_q >0 \\ ~~~~m_1+\cdots+m_q=r-k }} \! \! \!\!\!\!\!\!\!\!\!  a_{m_1}\cdots a_{m_q} \right)\cdot d_k+d_r,
		\end{alignat*}
		
		then we can write the formal expansion
		$$ \sum_{r=1}^\infty Z_r z^r = \left(\sum_{q=1}^\infty \left(\sum_{n=1}^\infty a_n z^n\right)^{\! \! \! q}\right)\cdot \left( \sum_{k=0}^\infty d_k z^k\right)+\sum_{r=1}^\infty d_r z^r. $$
		It is easy to see that all series on the right side of the equality converges for $z$ small enough and to compute explicitly
		$$ \sum_{r=1}^\infty Z_r z^r= \frac{z}{2 ~ \left(e^{-z/4}-e^{z/4}\right)}+1$$
		which is an even function of $z$.
		
		Since all the coefficients in the expansion given in \mbox{Lemma \ref{Lemdupli4}} are equal to $0$ we get that
		$$ \forall M \in \mathbb{N}, ~~~ B_N(z_N)= \frac{H}{2\cdot4^{H-1}} + O_u\left(N^{\theta_M}\right)+ o_u(1).$$
		Let us now write $$O_u \left(N^{\theta_M} \right) =O_u \left(N^{\alpha_N(M-1-2H)-M+2}\right) =O _u \left(z_N^{2H+1-M}N^{-M+2}\right).$$
		From \eqref{H_2} we have $$z_N^{-1}=o_u(N^{\overline{\alpha}}),$$ therefore if we choose $M$ large enough we have $$O_u\left(N^{\theta_M}\right)=o_u\left( N^{\overline{\alpha}(M-2H-1)-M+2} \right)=o_u(1).$$ Finally $$B_N(z_N)= \frac{H}{2\cdot4^{H-1}} + o_u\left(1\right) $$ and Lemma \ref{lemdupli} is proven.
	\end{proof}
	\subsubsection{Final steps of the proof} \label{subsec:prooftheorem}
	\begin{proof}[Proof of Theorem \ref{Thm:cylinder}] In the cylinder $S \times \mathbb{R}$ we now consider a number of parallel circles depending on N. Each circle bear again the same configuration of points. Precisely, we choose $$0<\beta< \gamma <1$$ and take $\lfloor N^\beta\rfloor $ circles at the heights $$\frac{k}{N^\gamma}, ~ k\in \{1,\cdots, \lfloor N^\beta\rfloor \} \ . $$ We put on the $k$-th of these circles $4N$ points $(P^k_i)_{i=1}^{4N}$ of coordinates $$\left(\frac{i}{4N},\frac{k}{N^\gamma} \right)_{i=1}^{4N} \ . $$ We associate to those points the usual coefficients $$c^k_i=(-1)^i$$ and consider
		
		\begin{alignat*}{1}
			Q_N=&\sum_{k,l=1}^{\lfloor N^\beta \rfloor} \sum_{i,j=1}^{4N}  c_i c_j ~ d^{2H}(P^k_i,P^l_j)\\
			=&\sum_{k=1}^{\lfloor N^\beta \rfloor}\sum_{i,j=1}^{4N} c_i c_j d^{2H}(P^k_i,P^k_j)+ \! \! \! \!\sum_{k,l=1,  k \neq l}^{\lfloor N^\beta \rfloor}\! \! \sum_{~~i,j=1}^{4N}  c_i c_j  d^{2H}(P^k_i,P^l_j)\\
			=& \lfloor N^\beta \rfloor A_N + \sum_{k,l=1,  k \neq l}^{\lfloor N^\beta \rfloor} B_N\left(z^{k,l}_N\right),
		\end{alignat*}
		
		with $$z^{k,l}_N= \frac{|k-l|}{N^{\gamma}}.$$ Let us observe that all the $z^{k,l}_N$ verify $$\frac{1}{N^{\gamma}}\leq z^{k,l}_N\leq \frac{\lfloor  N ^{\beta}\rfloor}{N^{\gamma}}$$ and recall that  $0<\beta< \gamma <1$, hence we can apply Lemma \ref{lemdupli}, since all $z^{k,l}_N$ verify  \eqref{H_1} together with \eqref{H_2} as long as we choose $\underline{\alpha},\overline{\alpha}$ such that $$0<\underline{\alpha}<\gamma<1$$ and $$0<\gamma-\beta<\overline{\alpha}<1 \ , $$ which is always possible.
		
		In the end we get that
		\begin{alignat*}{1}
			Q_N=& \lfloor N^\beta \rfloor A_N + \frac{\lfloor N^\beta \rfloor  \left( \lfloor N^\beta \rfloor -1 \right)}{2}\left(\frac{H}{2\cdot 4^{H-1}} + o(1)\right).
		\end{alignat*} Recall from Lemma \ref{Lemcercle} that $$  A_N \underset{N \rightarrow \infty}{\sim} \frac{N^{1-2H}}{4^{2H-1}} \sum_{p=0}^{\infty} \left[(2p)^{2H}-2(2p+1)^{2H}+(2p+2)^{2H}\right],$$
		\\ therefore if we choose $\beta > 1-2H$ we obtain
		\begin{equation}
			\label{aswewanted} Q_N \underset{N \rightarrow \infty}{\sim} \frac{H}{4^{H}} \cdot  N^{2\beta} \underset{N \rightarrow \infty}{\longrightarrow} + \infty \text{ as we wanted.}
		\end{equation}
		Let us remark that for every positive $\varepsilon$ the points $P_{i,N}$ belongs to $S \times ]0,\varepsilon[$ for $N$ large enough: Theorem \ref{Thm:cylinder} is proven.
	\end{proof}

	\section{Extension of the result to Riemannian products}  \label{sec:products}	\sectionmark{Extension to Riemannian products}
	Let us recall some facts about Riemannian products. Given two differential manifolds $M$ and $N$, the Cartesian product $M \times N$ has a natural structure of differential manifold. Furthermore for every $(p,q)$ in  $M \times N$, \begin{equation}T_{(p,q)}(M \times N)=T_p M  \times T_q N.\end{equation} For every $u\in T_{(p,q)}(M\times N)$ we will write $u=(u_M,u_N)$.
	The Riemannian product of two Riemannian manifolds $M$ and $N$ is the manifold $M \times N$ endowed with the product Riemannian metric, given for every $u,v\in T_{(p,q)}(M\times N)$ by \begin{equation}\langle u,v  \rangle_{M\times N}= \langle u_M, v_M \rangle_M + \langle u_N, v_N \rangle_N.\end{equation}
	The geodesics in the Riemannian product $M\times N$ are exactly the curves $g(t)=(m(t),n(t))$ with $m$ and $n$ geodesics in $M$ and $N$. The same is true for minimal geodesics. As a consequence we have the following equality between the geodesics distances :
	
	\begin{equation} \label{eq:geodesicdistanceproduct} d_{M\times N} ((p_1,q_1),(p_2,q_2))= \sqrt{d_M(p_1,p_2)^2+d_N(q_1,q_2)^2}.\end{equation}
		
	\begin{Thm} \label{Thm:extensionproducts}
		For every Riemannian manifolds $M$ and $N$ such that $M$ contains a minimal closed geodesic, the Riemannian product $M \times N$ has fractional index
		$$ \beta_{M\times N} =0.$$
	\end{Thm}
	\begin{proof} Let us consider $$\gamma : [0,2 \pi ]\rightarrow M$$ a minimal closed geodesic and $$g : [0,T] \rightarrow N$$ any minimal geodesic in $N$, which we choose to parametrise by arc-length. Since $\gamma$ is a minimal closed geodesic $\gamma([0,2\pi])$ is isometric to the circle of radius $\frac{L(\gamma)}{2\pi}$. In the same way, $g$ minimal geodesic implies that $g(]0,T[)$ is isometric to $]0,T[$. From \eqref{eq:geodesicdistanceproduct} we deduce that $$\gamma([0,2 \pi]) \times g(]0,T[) \subset M \times N$$ is isometric to the cylinder of radius $\frac{L(\gamma)}{2\pi}$ and height $T$. The fractional exponent of this cylinder is the same as $\beta_{\mathbb{S}^1 \times\left] 0,\frac{2 \pi T}{L(\gamma)} \right[}$ (see Remark \ref{Rem:homo}), which is null from Theorem \ref{Thm:cylinder}. From Remark \ref{Rem:subspace} we deduce that $\beta_{M\times N}$ is also null.\qedhere
	\end{proof}
	\begin{Exe}For every $d \geq 2$ the $d$-dimensional flat torus $\mathbb{T}^d:=\underbrace{\mathbb{S}^1 \times \cdots \times \mathbb{S}^1}_{\text{d times}}$ has fractional index $0$.
	\end{Exe}
	\begin{Exe} For every $n \geq 1$, $\mathbb{S}^n \times \mathbb{R}$ has fractional index $0$.
	\end{Exe}
	
	\section{Perturbation of the product distance} \label{sec:perturbation}
	\sectionmark{Perturbation of the result}
	In the following section we look at $\mathbb{S}^1\times ]0,\varepsilon[$ endowed with a distance which is not the product distance, but which converges to $d_{\mathbb{S}^1 \times \mathbb{R}}$ as $z \in ]0,\varepsilon[$ is close to $0$. 
	We give in Theorem \ref{perturbation} a bound on the fractional index in this case, which depends on some rate of convergence towards the cylinder distance. In Section \ref{subsec:examples_revo} we consider  some surfaces of revolution as examples.
	
	\begin{Thm} \label{perturbation}
		Let us consider a distance $d'$ on $\mathbb{S}^1 \times ]0,\varepsilon[$ and denote by $E'$ the resulting metric space. We define for very $h\in]0,\varepsilon[$
		$$ \Delta(h):=\sup_{z_1,z_2 \leq h} ~\sup_{\theta_1,\theta_2 \in \mathbb{S}^1} \left|d'[(\theta_1,z_1),(\theta_2,z_2)]-d[(\theta_1,z_1),(\theta_2,z_2)] \right|.$$
		where $d$ denotes the classical distance on the cylinder. We call $$\delta_{E'}:=\sup \left\{\delta>0,~ \Delta(h) \! \! \underset{ h \rightarrow 0^+\! \!}{=}  \! O\left( h^\delta \right) \right\}.$$
		If $\delta_{E'}$ is finite we obtain that the fractional index of $E'$ $\beta_{E'}$ verifies $$\beta_{E'}\leq\frac{3}{\delta_{E'}+1},$$
		and if $\delta_{E'}=+\infty,$ $$ \beta_{E'}=0.$$
	\end{Thm}
	\begin{Rem} \label{uninteresting}
		The result is uninteresting if $0<\delta_{E'}\leq 2$ since from Remark \ref{Rem:subspace} it is then clear that $$\beta_{E'}\leq 1 \leq \frac{3}{\delta_{E'}+1}.$$ Indeed $\mathbb{S}^1 \times \{0 \}\subset E'$ is isometric to $\mathbb{S}^1$, which has fractional index $1$ (see Istas \cite{istas2011manifold}).
	\end{Rem}
	\begin{proof} Let us assume there exists $\delta>0$ such that $$\Delta(h) \! \! \underset{ h \rightarrow 0^+\! \!}{=}  \! O\left( h^\delta \right)$$ which is true whether $\delta_{E'}$ is finite or $+\infty$. With the above remark the theorem is obvious for $0<\delta_{E'}\leq 2$. From now on we assume that $\delta_{E'} >2$. We consider $\delta < \delta_{E'}$ and \begin{equation} \label{condH} \frac{1}{2}> H>\frac{3}{2(\delta+1)}.\end{equation}
		We now apply the exact scheme of the proof of Theorem \ref{Thm:cylinder}, with the new distance $d'$. Let us recall that the proof of Theorem \ref{Thm:cylinder} lies on the existence of $\beta$ and $\gamma$ such that \begin{equation} \label{condab1} 1-2H < \beta < \gamma < 1.\end{equation} Our assumption \eqref{condH} allows us to choose $\beta$ and $\gamma$ such that besides \eqref{condab1} we have
		\begin{equation} \label{condab2} \delta(\beta-\gamma) < 2H-3, \end{equation} which will be useful later.
		With the notations of Section \ref{subsec:prooftheorem} we consider
		\begin{alignat}{1}
			Q'_N :&= \sum_{k,l=1}^{\lfloor N^\beta \rfloor} \sum_{i,j=1}^N c_j c_j  \left[d' (P_i^k,P_j^l)\right]^{2H}\\
			&= \sum_{k,l=1}^{\lfloor N^\beta \rfloor} \sum_{i,j=1}^N c_j c_j \left[d (P_i^k,P_j^l)+d' (P_i^k,P_j^l)-d (P_i^k,P_j^l) \right]^{2H} \nonumber \\
			&=\sum_{k,l=1}^{\lfloor N^\beta \rfloor} \sum_{i,j=1}^N c_j c_j \left[ d \left(P_i^k,P_j^l\right)\right]^{2H} \left[1+\frac{d' (P_i^k,P_j^l)-d (P_i^k,P_j^l)}{d (P_i^k,P_j^l)} \right]^{2H}. \label{Q'N}
		\end{alignat}
		As  the maximum altitude of all points considered is $\frac{\left\lfloor N^\beta \right\rfloor}{N^\gamma}$, using \eqref{condab2} we obtain for every $i,j,k,l$
		\begin{equation} \label{EO} \left| d' (P_i^k,P_j^l)-d (P_i^k,P_j^l) \right| \leq \Delta\left(N^{\beta-\gamma}\right)=O\left( N^{\delta (\beta-\gamma)}\right)=o\left(N^{2H-3}\right),\end{equation}
		moreover \begin{equation} \label{min_d} d (P_i^k,P_j^l)\geq \frac{1}{4N} \end{equation}
		so that $\displaystyle{\frac{d' (P_i^k,P_j^l)-d (P_i^k,P_j^l)}{d (P_i^k,P_j^l)}}$ tends towards $0$ as $N$ goes to infinity  for every $i,j,k,l$. 
		
		Taylor expansions yields
		$$ Q'_N=Q_N + O \! \left( \sum_{k,l=1}^{\lfloor N^\beta \rfloor} \sum_{i,j=1}^N c_j c_j 2H d^{2H-1} (P_i^k,P_j^l)~ \left( d'\left(P_i^k,P_j^l\right)-d\left(P_i^k,P_j^l\right) \right) \right). $$
		
		We compute
		\begin{alignat*}{1}
			|Q'_N-Q_N|&  = O \left( \sum_{k,l=1}^{\lfloor N^\beta \rfloor} \sum_{i,j=1}^N H d^{2H-1}(P_i^k,P_j^l) ~ \Delta\left(N^{\beta-\gamma}\right) \right)\\
			& = O \left( \lfloor N^\beta \rfloor^2 (4N)^2 \left(\frac{1}{N}\right)^{2H-1}\right)\Delta\left(N^{\beta-\gamma}\right),
		\end{alignat*}
		using \eqref{min_d} again and \eqref{EO} we obtain
		\begin{equation} \label{est}
			|Q'_N-Q_N|=O \left(N^{2\beta + 2 - 2H +1} \right) o \left( N^{2H-3} \right)= o \left( N^{2\beta}\right).
		\end{equation}
		Now given \eqref{condab1} and because $H<1/2$ we still have (see \eqref{aswewanted})
		$$ Q_N \underset{N \rightarrow \infty}{\sim} \frac{H}{4^{H}} \cdot  N^{2\beta},$$
		hence $$\displaystyle{ Q'_N \underset{N \rightarrow \infty}{=} \frac{H}{4^{H}} \cdot  N^{2\beta} + o\left(N^{2\beta}\right)}$$
		is positive for $N$ large enough, which implies that $\left(d'\right)^{2H}$ is not negative definite and therefore $\beta_{E'}<2H$. Since this is true for every $\delta <\delta_C$ and every $H>\frac{3}{2(\delta+1)}$, the theorem is proven.
	\end{proof}
	We now turn to the case of some Riemannian surfaces in a given chart.
	\begin{Thm} \label{Prosurf}
		Let $I$ be an open real interval such that there exists $\varepsilon>0$, $]0,\varepsilon[ \subset I$ and consider the case where $E'$ is  $\mathbb{S}^1 \times I$ endowed with the Riemannian metric
		$$ \langle~,~\rangle'=(1+f_1(\theta,z))d\theta^2+(1+f_2(\theta,z))dz^2,$$
		with $f_1$ and $f_2$ $C^\infty$ functions with values in $]-1,+\infty[$.
		
		Let us assume that the Riemannian manifold $E'$ is complete, and that
		\begin{equation} \label{eq:hypofinite} \sup_{P,Q \in \mathbb{S}^1 \times ]0, \varepsilon [} \sup \left\{ \max\left(\int_{\gamma_{d'}}|d\theta|, \int_{\gamma_{d'}}  |dz|\right), \begin{aligned}  \gamma_{d'} \text{ minimal geodesic in}\\ \text{$E'$ between } P \text{ and } Q \end{aligned} \right\} < \infty \end{equation} For every $h \in  I$ we define
		$$z^+(h):=\sup_{P,Q\in \, \mathbb{S}^1 \! \times ]0,h]} \inf \,\left\{
		\begin{aligned}
		&\max_t(z(t)) \text{ such that } t\mapsto (\theta(t),z(t)) \text{ is a }\\ &\text{minimal}  \text{ geodesic in $E'$ between } P \text{ and } Q
		\end{aligned}
		\right\},$$
		$$z^-(h):=\sup_{P,Q\in \, \mathbb{S}^1 \! \times ]0,h]} \sup \,\left\{
		\begin{aligned}
		&\min_t(z(t)) \text{ such that } t\mapsto (\theta(t),z(t)) \text{ is a }\\ &\text{minimal}  \text{ geodesic in $E'$ between } P \text{ and } Q
		\end{aligned}
		\right\},$$
		$$F_1(h):=\sup_{z\in ]z^-(h),z^+(h)[} \, \, \max _{\theta \in \, \mathbb{S}^1}\sqrt{|f_1(\theta,z)|}, ~\delta_1:=\sup \left\{\delta>0, F_1(h)\! \! \underset{ h \rightarrow 0^+\! \!}{=} O\left( h^\delta \right) \right\}, $$ $$F_2(h):=\sup_{z \in ]z^-(h),z^+(h)[} \, \, \max _{\theta \in \, \mathbb{S}^1}\sqrt{|f_2(\theta,z)|}, ~ \delta_2 :=\sup \left\{\delta>0, F_2(h) \! \! \underset{ h \rightarrow 0^+\! \!}{=}  O\left( h^\delta \right) \right\}.$$
		If $\min(\delta_1,\delta_2)$  is finite we have
		$$\beta_{E'} \leq \frac{3}{\min\left( \delta_1,\delta_2\right)+1},$$
		and if $\delta_1=\delta_2=+\infty$,
		$$ \beta_{E'}=0. $$
	\end{Thm}
	
	\begin{proof}
		Let us consider $P_1=(\theta_1,z_1)$ and $P_2=(\theta_2,z_2)$ in $\mathbb{S}^1\times ]0,\varepsilon[$. Let us denote by $\gamma_{d'}$ a minimal geodesic between $P_1$ and $P_2$ in $E'$, and by $\gamma_d$ a minimal geodesic between the same points in the cylinder endowed with its classical distance $d$. We also set
		$$C=\sup_{P_1,P_2 \in \, \mathbb{S}^1 \times ]0,\varepsilon[}  \, \sup_{\gamma_d,\gamma_{d'}} ~~ \max \left\{\int_{\gamma_{d}} \! \! |d\theta|, \int_{\gamma_{d}} \! \! |dz|, \int_{\gamma_{d'}} \! \! |d\theta|,\int_{\gamma_{d'}} \! \! |dz|\right\}.$$
		Let us notice that $C$ is finite from hypothesis \eqref{eq:hypofinite}. Indeed the curves $\gamma_d$ are minimal geodesics between points in $\mathbb{S}^1 \times ]0,\varepsilon[$, hence $\int_{\gamma_{d}} \! \! |dz| < \varepsilon$ and $\int_{\gamma_{d}} \! \! |dz| < \pi$.
		We now assume that  $z_1,z_2\leq h$ and compute
		\begin{alignat*}{1}
			d'(P_1,P_2) & =\int_{\gamma_{d'}} \! \!  \left(\left\langle \gamma_{d'}',\gamma_{d'}'\right\rangle'\right)^{1/2} \leq \int_{\gamma_{d}} \! \!  \left(\left\langle \gamma_{d}',\gamma_{d}'\right\rangle'\right)^{1/2}\\
			&= \int_{\gamma_{d}} \Bigl((1+f_1(\theta,z))d\theta^2+(1+f_2(\theta,z))dz^2\Bigr)^{1/2} \\
			& \leq \int_{\gamma_{d}} \Bigl((1+\max|f_1 \circ \gamma_d | )d\theta^2+(1+\max|f_2 \circ \gamma_d |) dz^2\Bigr)^{1/2}, \\
			& \text{using twice }(a+b)^{1/2}\leq a^{1/2}+b^{1/2} \text{ for } a,b>0:\\
			& \leq \int_{\gamma_{d}} \! \left(d\theta^2+dz^2\right)^{1/2} +\max|f_1 \circ \gamma_d |^{1/2} \int_{\gamma_{d}} \! |  d\theta| + \max|f_2 \circ \gamma_d |^{1/2} \int_{\gamma_{d}} \! |dz| \\
			& \leq d(P_1,P_2)  + C \left(\max|f_1 \circ \gamma_d |^{1/2}+\max|f_2 \circ \gamma_d |^{1/2}\right),\\
		\end{alignat*}
		from which we deduce
		\begin{equation} d'(P_1,P_2) \leq d(P_1,P_2) + C(F_1(h)+F_2(h)).\end{equation}

		In a similar way and with $$f_{i}^-(\theta,z) := -\min(f_{i}(\theta,z),\, 0):$$
		\begin{alignat*}{1}
			d'(P_1,P_2) &= \int_{\gamma_{d'}} \Bigl((1+f_1(\theta,z))d\theta^2+(1+f_2(\theta,z))dz^2\Bigr)^{1/2}\\
			& \geq \int_{\gamma_{d'}} \Bigl((1-f_1^{-}(\theta,z))d\theta^2+(1-f_2^{-}(\theta,z))dz^2\Bigr)^{1/2}\\
			& \text{using }(a-b)^{1/2}\geq a^{1/2}-b^{1/2} \text{ for } a>b>0:\\
			& \geq  \int_{\gamma_{d'}} \! \!\left(d\theta^2+dz^2 \right)^{1/2}-\max|f_1 \circ \gamma_{d'}|^{1/2}\int_{\gamma_{d'}} \! \! |d\theta| - \max|f_2 \circ \gamma_{d'}|^{1/2} \int_{\gamma_{d'}} \! \! |dz|\\
			& \geq  \int_{\gamma_{d}} \left(d\theta^2+dz^2 \right)^{1/2}- C\left(\max|f_1 \circ \gamma_d'|^{1/2}+\max|f_2 \circ \gamma_d'|^{1/2}\right),
		\end{alignat*}
		hence \begin{equation} d'(P_1,P_2) \geq d(P_1,P_2) + C(F_1(h)+F_2(h)). \end{equation}
		Finally for every $P_1=(\theta_1,z_1)$ and $P_2=(\theta_2,z_2)$ with $z_1,z_2 \leq h$ we have  $$ |d(P_1,P_2)-d'(P_1,P_2)| \leq C \left(F_1(h)+F_2(h)\right),$$ hence $$\Delta(h)\leq  C \left(F_1(h)+F_2(h)\right).$$
		This implies that $\delta_{E'}$ (defined in Theorem \ref{perturbation}) is such that $$\delta_{E'} \geq \min(\delta_1,\delta_2),$$
		and we apply Theorem \ref{perturbation} to conclude.\qedhere
	\end{proof}
	\begin{Rem} \label{Rem:boundedspacescheckassumption} Assumption \eqref{eq:hypofinite} is for example verified if $E'$ a metric space of finite diameter and $f_1$ and $f_2$ are bounded below by $m>-1$.  Indeed for every $P,Q$ in $\mathbb{S}^1 \times ]0, \varepsilon [$ and $\gamma_{d'}$ a minimal geodesic from $P$ to $Q$ in $E'$ we have
		$$ \int_{\gamma_{d'}} \left((1+f_1(\theta,z))d\theta^2+(1+f_2(\theta,z))dz^2\right)^{1/2} = d(P,Q),$$
		hence $$ \int_{\gamma_{d'}} \left((1+f_1(\theta,z))d\theta^2\right)^{1/2} \leq d(P,Q),$$
		from which we deduce
		$$ \int_{\gamma_{d'}} |d\theta| \leq \frac{d(P,Q)}{\inf(1+f_1(\theta,z))^{1/2}}.$$
		The same argument gives	$$ \int_{\gamma_{d'}} |dz| \leq \frac{d(P,Q)}{\inf(1+f_2(\theta,z))^{1/2}}.$$
	\end{Rem}
	\begin{Rem} Let $S$ be a complete, orientable Riemannian manifold of dimension $2$, with
		$$\gamma : [0,2\pi] \rightarrow S $$ a minimal closed geodesic. Without loss of generality (see Remark \ref{Rem:homo}) we assume that the minimal geodesic has length $L(\gamma)=2\pi$ and is parametrised by arc-length. If we choose a $C^\infty$ vector field $v$ along $\gamma$ such that for every $\theta$, $\|v(\theta)\|_S=1$ and $\langle v(\theta), \gamma'(\theta) \rangle_S =0$, and define
			\begin{alignat*}{1} \Phi : \mathbb{S}^1 \times \mathbb{R} & \rightarrow S\\
				(\theta,z) &\mapsto Exp_{\gamma(\theta)}(z v(\theta)),
		\end{alignat*}
		it is possible to check that there exists $\varepsilon>0$ such that the restriction of $\Phi$ to $\mathbb{S}^1\times ]-\varepsilon, \varepsilon [$  is a $C^\infty$ diffeomorphism onto its image $$\mathcal{V}_\varepsilon=\Phi(\mathbb{S}^1\times ]-\varepsilon, \varepsilon [).$$ Furthermore $\mathcal{V}_\varepsilon$ is a neighbourhood of $\gamma$. For every $p\in \mathcal{V}_\varepsilon$ we get the coordinates $(\theta,z)=\Phi^{-1}(p)$, and one can check that the inner product of $S$ is given by
		$$ \langle ~, ~ \rangle_S = \left(1+f_1(\theta,z)\right)d\theta^2 + dz^2,$$
		where $f_1$ is a $C^\infty$ function with values in $]-1,+\infty[$ such that $f_1(\theta,0)=0$ for every $\theta$.
		
		However it is not possible to apply Theorem \ref{Prosurf} without global assumptions on $S$. Indeed imposing some conditions on the inner product in a neighbourhood of $\gamma$ is not enough to control the geodesic distance, as minimal geodesics between points close to $\gamma$ may take values in the whole of $S$ (see Remark \ref{Rem:submanifold}).
		
		 In this case we need that all the minimal geodesics between points close enough to $\gamma$ take values in $\mathcal{V}_\varepsilon$, in order to have $z^+(h)$ and $z^-(h)$ properly defined. Furthermore if we don't have $\lim\limits_{h\rightarrow 0} z^+(h) = \lim\limits_{h\rightarrow 0} z^-(h)=0 $, we don't have $$\lim\limits_{h\rightarrow 0} F_1(h)= \lim\limits_{h\rightarrow 0} F_2(h)=0,$$  hence $\delta_1=\delta_2=-\infty$ and Theorem \ref{Prosurf} claims nothing. In the next section we consider revolution surfaces with increasing generating function and apply Theorem \ref{Prosurf}.
	\end{Rem}
		
	\subsection{Some surfaces of revolution as examples} \label{subsec:examples_revo}
	In all that follow, we consider a $\mathcal{C}^\infty$ function $r:\mathbb{R}^+ \rightarrow \mathbb{R}^+_*$ and we  call \emph{the surface of revolution with generating function $r$} the surface $\Gamma$ of $\mathbb{R}^3$ admitting the parametrisation
	
	\begin{equation} \label{pararotatio}  X_{\Gamma} : (\theta, z)\mapsto\left( \begin{array}{c}r(z)\cos(\theta) \\ r(z) \sin(\theta) \\ z\end{array}\right) .\end{equation}
	
	\begin{Lem} \label{Lem:Clair} Let $\Gamma$ be a surface of revolution with generating function $r$. If $r$ is increasing, for every geodesic
		\begin{alignat*}{1}
			g : [0,T] &\rightarrow \Gamma \\
			t &\mapsto (\theta(t),z(t))
		\end{alignat*}
		and for every $t\in [0,T],$ $$z(t) \leq \max(z(0),z(T)).$$
		
		In particular for every $h\geq 0$, $$z^+(h)=h.$$
	\end{Lem}
	\begin{proof} We will use Clairaut's relation (see \cite{docarmos}) which states that along a given geodesic of a surface of revolution
		
		\begin{equation} \label{Clairaut} r(z(t))\cos (\varphi(t))=const.~,\end{equation}
		
		where $\varphi(t)\in [0,\pi/2]$ is the acute, nonoriented angle that makes the geodesic with the parallel that intersects it at $t=0$.
		
		Since any geodesic is differentiable, so is $t \mapsto z(t)$. Let us assume that $z(t)$ has a global maximum in $t_0 \in ]0,T[$ and that there exists $t_1 \in ]0,T[$ such that $z'(t_1)\neq 0$. Since $z(t_0)$ is a maximum we have $z'(t_0)=0$, which is equivalent to $\varphi(t_0)=0$. Because $z'(t_1)\neq 0$, $\varphi(t_1)\in ]0,\pi/2]$.  We have
		$$ \cos(\varphi(t_1))<\cos(\varphi(t_0))=1.$$
		Using $r$ increasing and $z(t_1)\leq z(t_0)$ maximum, we obtain
		
		$$ r(z(t_1))\cos(\varphi(t_1))< r(z(t_0))\cos(\varphi(t_0)),$$
		which contradicts Clairaut's relation \eqref{Clairaut}.
		
		In the end, either $z'(t)=0$ for every $t\in]0,T[$, which means $z(t)=const.$ and the result is clear, either the global maximum of $z$ over $[0,T]$ (which exists since $z$ is continuous) is reached in $t=0$ or $t=T$. We have proven for every geodesic $t\mapsto (\theta(t),z(t))$ that $$\forall t \in [0,T], ~ z(t) \leq \max(z(0),z(T)).$$ Given the definition of $z^+$ (see Theorem \ref{Prosurf}) it is clear that $z^+(h)=h$ for every $h$. The lemma is proven.
	\end{proof}

	\begin{Thm} \label{Thm:revolution} Let $\Gamma$ be a revolution surface with $\mathcal{C}^\infty$ generating function $r$ such that $r$ is increasing.
		If there exists $p>4$ and $c\in \mathbb{R}$ such that \begin{equation} \label{eq:hypopetito} r(z)-r(0)\underset{z\rightarrow 0}{=}cz^p+o\left(z^p \right) \end{equation}
		then $$\beta_\Gamma \leq\frac{6}{p+2}.$$
		If for every $p\in \mathbb{N}$ \begin{equation} \label{eq:hypoplate} r(z)-r(0)\underset{z\rightarrow 0}{=} o\left(z^p \right), \end{equation}
		then $$\beta_\Gamma=0.$$
	\end{Thm}
	\begin{proof} 	Let us assume that $r(0)=1$. This assumption is done without loss of generality, as we may consider an homothety of a general surface of revolution $\Gamma$ to have $r(0)=1$, without changing the fractional index  (see again Remark \ref{Rem:homo}).
		
		We compute $$\frac{\partial X_\Gamma}{\partial \theta} =  \left( \begin{array}{c} -r(z)\sin(\theta) \\ r(z) \cos(\theta) \\  0 \end{array} \right) ,$$ $$\frac{\partial X_\Gamma}{\partial z} = \left( \begin{array}{c}r'(z)\cos(\theta) \\ r'(z) \sin(\theta) \\ 1\end{array}\right),$$ and deduce the coefficients of the first fundamental form of $\Gamma$ :
		
		\begin{alignat*}{1}E_\Gamma = & \left\langle \frac{\partial X_\Gamma}{\partial \theta}, \frac{\partial X_\Gamma}{\partial \theta} \right\rangle_{\! \! \mathbb{R}^3} = r^2(z),\\
			F_\Gamma = & \left\langle \frac{\partial X_\Gamma}{\partial \theta} , \frac{\partial X_\Gamma}{\partial z} \right\rangle_{\! \! \mathbb{R}^3} = 0,\\
			G_\Gamma = & \left\langle \frac{\partial X_\Gamma}{\partial z} , \frac{\partial X_\Gamma}{\partial z} \right\rangle_{\! \! \mathbb{R}^3}= r'(z)^2+1 .\end{alignat*}
		
		We get the expression of the Riemannian metric
		
		\begin{equation} \label{eq:innerproductsurface} \langle~,~ \rangle_\Gamma=r^2(z) d\theta^2 + \left(1+r'(z)^2\right)dz^2. \end{equation}
		
		Let us now fix a positive $\varepsilon$ and apply Theorem \ref{Prosurf} to $E'=\mathbb{S}^1\times ]0,\varepsilon[$ endowed with the inner product $\langle~,~ \rangle_\Gamma$. It is clear that $E'$ is isometric to the Riemannian manifold $\Gamma_{\varepsilon}:=X_\Gamma\left(\mathbb{S}^1\times ]0,\varepsilon[\right)$ endowed with $\langle~,~ \rangle_\Gamma $.
		
		Let us now check assumption \eqref{eq:hypofinite}, using Remark \ref{Rem:boundedspacescheckassumption}. It is clear that the Riemannian manifold $\Gamma_{\varepsilon}=X_\Gamma\left(\mathbb{S}^1\times ]0,\varepsilon[\right)$ endowed with $\langle~,~ \rangle_\Gamma $ is a metric space of finite diameter. We have \begin{alignat*}{1}f_1(\theta, z) &=  r^2(z)-1, \\ f_2(\theta,z) &=  r'(z)^2.   \end{alignat*}
		Since $r(0)=1$ and $r$ increasing we have $f_1(\theta,z) \geq 0 > -1$ and clearly $f_2(\theta,z) \geq 0 > -1.$ From Remark \ref{Rem:boundedspacescheckassumption}, assumption \eqref{eq:hypofinite} is verified.
		
		Recall that $z^+(h)=h$ from Lemma \ref{Lem:Clair}, and clearly $z^-(h)=0$. Since $f_1$ and $f_2$ do not depend on $\theta$ we get
		\begin{alignat*}{1} 
			F_1(h) & = (r^2(h)-1)^{1/2},\\
			F_2(h) & = \max_{z\in[0,h] } |r'(z)|.  \end{alignat*}
	
	Now under assumption \eqref{eq:hypopetito}, it is clear that we have $$F_1(h)=(2c)^{1/2}h^{p/2}+o\left(h^{p/2}\right),$$
	hence \begin{equation}\label{eq:majoration_delta1} \delta_1 \geq p/2. \end{equation}
	
	Since $r'$ is continuous, there exists $\hat{z}(h)\in [0,h]$ such that $$F_2(h)= \max_{z\in[0,h] } |r'(z)|=|r'(\hat{z}(h))|.$$
	
	From \eqref{eq:hypopetito} we get $$F_2(h)=|r'(\hat{z}(h))|=pc\hat{z}(h)^{p-1}+o\left((\hat{z}(h)^{p-1}\right),$$
	and since $x\mapsto x^{p-1}$ is decreasing,
	$$F_2(h) \leq pch^{p-1}+o\left(h^{p-1}\right),$$
	hence \begin{equation}\label{eq:majoration_delta2} \delta_2 \geq p-1. \end{equation}
	
	Since $p>4$ from \eqref{eq:majoration_delta1} and \eqref{eq:majoration_delta2} we have  $\min(\delta_1,\delta_2) \geq p/2 $, and applying Theorem \ref{Prosurf} we get
	\begin{equation} \label{eq:conclu1} \beta_{\Gamma_\varepsilon} \leq \frac{3}{p/2+1}=\frac{6}{p+2}. \end{equation}
	
	Under assumption \eqref{eq:hypoplate} instead of  \eqref{eq:hypopetito}, the same reasoning gives
	\begin{equation} \label{eq:conclu2} \Gamma_{\varepsilon} =0.\end{equation}
	
	Now since from Lemma \ref{Lem:Clair} we have $z^+(h) =h$, every geodesic in $\Gamma$ between points of $\Gamma_\varepsilon= X_\Gamma(\mathbb{S}^1\times [0,\varepsilon])$ stays in $\Gamma_{\varepsilon}$, hence the geodesic distances $d_{\Gamma}$ and $d_{\Gamma_\varepsilon}$ coincide on $\Gamma_\varepsilon$. This allows to extend the conclusions \eqref{eq:conclu1} and \eqref{eq:conclu2} from $\beta_{\Gamma_{\varepsilon}}$ to $\beta_{\Gamma}$ (see Remark \ref{Rem:subspace}). 	\end{proof}

\begin{Rem} Notice that the assumption $p>4$ is not essential. However for lower values of $p$ the same reasoning gives a bound of $\beta_{\Gamma}$ greater than $1$, and we already know that $\beta_{\Gamma} \leq \beta_{\mathbb{S}^1}= 1$ since the parallel at height $z=0$ in $\Gamma$ is isometric to a circle (see Remark \ref{Rem:subspace}).
\end{Rem}
\begin{Exe} As we can see the bound obtained in Theorem \ref{Thm:revolution} is the smaller as the order of the contact between the surface $\Gamma $ and the cylinder is great. In particular if we consider the generating function $r(z)=1+e^{-\frac{1}{z}}$ whose all derivatives at $z=0$ are zero, we obtain
	\begin{equation}\beta_{\Gamma}=0,\end{equation}
	which indicates that our argument does not depend on the product structure of the index space.
\end{Exe}
	\section{Gromov-Hausdorff discontinuity of $E \mapsto \beta_E$} \label{sec:discontinuity}
	We recall that it is possible to endow the set $\mathcal{M}$ of all isometry classes of compact metric spaces with the Gromow-Hausdorff distance $d_{GH}$.
	
	Given two closed sets $A,B$ in a metric space $(E,d_E)$, the \emph{Hausdorff distance} between $A$ and $B$ is
	
	\begin{equation} d_{\mathcal{H}} (A,B) := \max \{ \sup_{x\in A} d_E(x,B) , \sup_{y \in B} d_E(y,A) \}.\end{equation}

	We now give the definition of the \emph{Gromov-Hausdorff distance} between two isometry classes of compact metric spaces $\bar{E}$ and $\bar{F}$,
	\begin{equation} d_{\mathcal{GH}} (\bar{E}, \bar{F}):= \inf_{i,j} d_\mathcal{H} (i(E),j(F)), \end{equation}
	where  $E$ and $F$ are any two representatives of $\bar{E}$ and $\bar{F}$, $i$ and $j$ run through all isometrics embeddings of $E$ and $F$  into any ambient metric space $(X,d)$, and $d_{\mathcal{H}}$ denotes the Hausdorff distance on closed sets of $(X,d)$.
	
	It is known that $\left(\mathcal{M},d_{\mathcal{GH}}\right) $ is a metric space (see \cite{bbi}).
	\begin{Thm} \label{Thm:Gromovdiscontinuity} The map \begin{alignat*}{1}\left(\mathcal{M},d_{\mathcal{GH}}\right) & \rightarrow \mathbb{R}^+ \\ E & \mapsto \beta_E\end{alignat*} is not continuous at $E=\mathbb{S}^1$. 
	\end{Thm}
	\begin{proof} Let us consider $\mathbb{S}^1 \times [0,\varepsilon]$ endowed with the Riemannian product metric \eqref{eq:cylinder_metric}, which is nothing more than $\mathbb{S}^1\times [0,\varepsilon]$ endowed with the restriction of $d_{\mathbb{S}^1\times\mathbb{R}}$. (All that we say about $\mathbb{S}^1 \times ]0,\varepsilon [$ in Remark \ref{Rem:coincides} is true for $\mathbb{S}^1 \times [0,\varepsilon]$.)		
		
		It is clear that the isometry class of  $\mathbb{S}^1 \times [0,\varepsilon]$ converges towards the isometry class of $\mathbb{S}^1$ regarding the Gromov-Hausdorff distance. Indeed if we denote by $C_\varepsilon=i(\mathbb{S}^1 \times [0,\varepsilon])$ the canonical embedding of $\mathbb{S}^1\times [0,\varepsilon]$ in $\left(\mathbb{S}^1 \times \mathbb{R},d\right),$
		\begin{alignat*}{1}
			d_{\mathcal{GH}} \left( \mathbb{S}^1, \mathbb{S}^1 \times [0,\varepsilon]\right) & \leq d_{\mathcal{H}}\left( C_0, C_\varepsilon \right)\\
			&  = \max \left(\sup_{y\in C_\varepsilon} d(x,C_0), \sup_{y \in C_0} d(y,C_\varepsilon) \right) \\
			& =\max(\varepsilon, 0)=\varepsilon. 
		\end{alignat*}
		Recall that $\beta_{\mathbb{S}^1}=1$. From Theorem \ref{Thm:cylinder} we know that for every $\varepsilon >0$, $\beta_{\mathbb{S}^1 \times [0,\varepsilon]}=0$. The discontinuity at $E=\mathbb{S}^1$ is proven.
	\end{proof}
	
		\section*{Acknowledgements} The author is indebted to Serge Cohen who carefully advised him during this work. He is also grateful to Benoit Huou for fruitful discussions.
	\clearpage
\bibliographystyle{plain}
\bibliography{Nil}
\end{document}